\documentclass[11pt]{amsart}
\usepackage{enumerate,indentfirst}
\usepackage{amssymb,amsfonts,amsmath,mathrsfs}
\usepackage{mathptmx}
\usepackage[colorlinks=true,linkcolor=blue,citecolor=blue]{hyperref}
\usepackage{geometry}
\newtheorem{theorem}{Theorem}[section]
\newtheorem{proposition}[theorem]{Proposition}
\newtheorem{corollary}[theorem]{Corollary}

\newtheorem{example}[theorem]{Example}

\numberwithin{equation}{section}
\newcommand{\SU}{\textrm{SU}(2)}
\def\jp#1{{\left\langle{#1}\right\rangle}}
\begin{document}
\title{Regularity of solutions to a Vekua-type equation on compact Lie groups}


\author[Wagner A.A. de Moraes]{Wagner Augusto Almeida de Moraes}
\address{
	Ghent University, 
	Department of Mathematics: Analysis, Logic and Discrete Mathematics,
	Krijgslaan 281, Building S8, 
	B 9000 Ghent, Belgium. 
	ORCID 0000-0002-5624-7374}
	\email{wagneramat@gmail.com, wagneraugusto.almeidademoraes@ugent.be}

\thanks{This study was financed in part by the Coordenação de Aperfeiçoamento de Pessoal de Nível Superior - Brasil (CAPES) - Finance Code 001 and by the Methusalem programme of the Ghent University Special Research Fund (BOF) (Grant number 01M01021).}


\date{\today}


\begin{abstract}
We present sufficient conditions to have global hypoellipticity for a class of Vekua-type operators defined on a compact Lie group. When the group has the property that every non-trivial representation is not self-dual we show that these sufficient conditions are also necessary. We also present results about the global solvability for this class of operators.
\end{abstract}

\maketitle

\section{Introduction and preliminaries results}

In this paper, we are interested in the study of global properties for a class of operators defined on a compact Lie group $G$. We will present necessary and sufficient conditions to obtain the global hypoellipticity and the global solvability of the operator $P: C^\infty(G) \to C^\infty(G)$ given by
\begin{equation}\label{operatorP}
Pu:= Xu-qu-p\overline{u},
\end{equation}
where $X$ is a normalized vector field on $G$ and $p,q \in \mathbb{C}$, with $p \neq 0$. The case $p=0$ was studied in \cite{KMR20b}. Notice that when $p \neq 0$ the operator $P$ is $\mathbb{R}$-linear but is not $\mathbb{C}$-linear. We say that $P$ is globally hypoelliptic if the conditions $u\in \mathcal{D}'(G)$ and $Pu \in C^\infty(G)$ imply $u\in C^\infty(G)$. 

The local solvability for classes of such operators was considered in \cite{Mez1} and \cite{Mez2} and the global solvability
on the torus $\mathbb{T}^2$ was studied in \cite{BDM14} in the case where $G=\mathbb{T}^2$.

Our results were inspired by the article \cite{BDM14} by A. Bergamasco, P. Dattori, and  A. Meziani, where the authors consider the case $G=\mathbb{T}^2$ and they use Fourier analysis to obtain their results. Since the Fourier analysis on compact Lie groups is well developed, we were able to give sufficient and necessary conditions for the global hypoellipticity of the operator $P$ given in \eqref{operatorP} and for the necessary condition we were led to assume that every non-trivial representation of $G$ is not self-dual. With this same assumption we have discussed the global solvability of $P$, but using a different notion of solvability from the one considered on \cite{BDM14}.

In \cite{RTW14}, M. Ruzhansky, V. Turunen, and J. Wirth have given sufficient conditions for the global hypoellipticity for a class of pseudo-differential operators on compact Lie groups in terms of their matrix-valued full symbols, that will be defined on \eqref{symbol}. This approach works well in the case where $p=0$ (see \cite{KMR20b}) because in this case, the operator $P$ is linear. When $p \neq 0$, even $P$ being a left-invariant operator, its symbol depends on $x\in G$, so we do not have the expected property $\widehat{Pu}(\xi) = \sigma_P(\xi) \widehat{u}(\xi)$, for every $[\xi] \in \widehat{G}$.
For this reason, our approach will be done by analyzing the Fourier coefficients of $Pu$ at $\xi$ and $\bar{\xi}$ to obtain a suitable expression that relates $Pu$ and $u$.

 In Section \ref{suf-GH} we present sufficient conditions for the global hypoellipticity of the operator $P$. In Section \ref{nec-GH} we present necessary conditions for the global hypoellipticity of the operator $P$ when $G$ has the property that its non-trivial representations are not self-dual. In Section \ref{sectionGS} we discuss the global solvability of the operator $P$ with the same additional hypothesis from Section \ref{nec-GH}. In Section \ref{operatoronsu2} we study the operator $P$ on $\SU$, which is a compact Lie group where every representation is self-dual.

Throughout this paper, we will use the notation and results about Fourier analysis on compact Lie groups  based on the book by M. Ruzhansky and V. Turunen \cite{livropseudo} to study the global hypoellipticity and the global solvability of the operator $P$ defined on \eqref{operatorP}.

Let $G$ be a compact Lie group and let $\textrm{Rep}(G)$  be the set of continuous irreducible unitary representations of $G$. By the compactness of $G$ every $\phi \in \textrm{Rep}(G)$ has finite dimension and it can be viewed as a matrix-valued function $\phi: G \to \mathbb{C}^{d_\phi\times d_\phi}$, where $d_\phi = \dim \phi$. We say that $\phi \sim \psi$ if there exists an unitary matrix $A\in \mathbb{C}^{d_\phi \times d_\phi}$ such that $A\phi(x) =\psi(x)A$, for all $x\in G$. We will denote by $\widehat{G}$ the quotient of $\textrm{Rep}(G)$ by this equivalence relation.
Given $\phi \in \textrm{Rep}(G)$, we define its conjugate representation $\overline{\phi}: G \to \mathbb{C}^{d_\phi\times d_\phi}$ as $\overline{\phi}(x) := \overline{\phi(x)}$ and we have that $\overline{\phi} \in \textrm{Rep}(G)$. When $\phi \sim \overline{\phi}$ we say that $\phi$ is a self-dual representation.

For $f \in L^1(G)$ we define the group Fourier transform of $f$ at $\phi \in \textrm{Rep}(G)$ is
\begin{equation*}
	\widehat{f}(\phi)=\int_G f(x) \phi(x)^* \, dx,
\end{equation*}
where $dx$ is the normalized Haar measure on $G$.
By the Peter-Wyel theorem, we have that 
\begin{equation}\label{ortho}
	\mathcal{B} := \left\{\sqrt{\dim \phi} \, \phi_{ij} \,; \ \phi=(\phi_{ij})_{i,j=1}^{d_\phi}, [\phi] \in \widehat{G} \right\},
\end{equation}
is an orthonormal basis for $L^2(G)$, where we pick only one matrix unitary representation in each class of equivalence, and we may write
\begin{equation*}
	f(x)=\sum_{[\phi]\in \widehat{G}}d_\phi {\textrm{Tr}}(\phi(x)\widehat{f}(\phi)).
\end{equation*}
and the Plancherel formula holds:
\begin{equation}
	\label{plancherel} \|f\|_{L^{2}(G)}=\left(\sum_{[\phi] \in \widehat{G}}  d_\phi \ 
	\|\widehat{f}(\phi)\|_{\mathtt{HS}}^{2}\right)^{{1}/{2}}=:
	\|\widehat{f}\|_{\ell^{2}(\widehat{G})},
\end{equation}
where 
\begin{equation*} \|\widehat{f}(\phi)\|_{\mathtt{HS}}^{2}={\textrm{Tr}}(\widehat{f}(\phi)\widehat{f}(\phi)^{*})=\sum_{i,j=1}^{d_\phi}  \bigr|\widehat{f}(\phi)_{ij}\bigr|^2.
\end{equation*}

For each $[\phi] \in \widehat{G}$, its matrix elements are eigenfunctions of the Laplace-Beltrami operator $\mathcal{L}_G$ correspondent to the same eigenvalue $-\nu_{[\phi]}$, where $\nu_{[\phi]} \geq 0$. Thus
$$
-\mathcal{L}_G \phi_{ij}(x) = \nu_{[\phi]}\phi_{ij}(x), \quad \textrm{for all } 1 \leq i,j \leq d_\phi,
$$
and we will denote by
$$
\jp \phi := \left(1+\nu_{[\phi]}\right)^{1/2}
$$
the eigenvalues of $(I-\mathcal{L}_G)^{1/2}.$

It is possible to characterize smooth functions and distributions on $G$ by relations between their Fourier coefficients at $\phi$ and $\jp{\phi}$. Precisely, we have

\begin{enumerate}
	\item  $f \in C^\infty(G)$ if and only if for each $N>0$, there exists $C_N > 0$ such that
	\begin{equation}\label{cinfty}
			|\widehat{f}(\phi)_{ij}| \leq C_N \jp{\phi}^{-N}, 
	\end{equation}
	for all $[\phi] \in \widehat{G}$ and $1\leq i,j \leq d_\phi$.
	\item $u \in \mathcal{D}'(G)$ if and only if there exist $C, \, N > 0$ such that 
\begin{equation}\label{distr}
	|\widehat{u}(\phi)_{ij}| \leq C \jp{\phi}^{N},
\end{equation}
	for all $[\phi] \in \widehat{G}$ and $1\leq i,j \leq d_\phi$, where $\widehat{u}(\phi)_{ij} := \jp{u,\overline{\phi_{ji}}}$.
\end{enumerate}

For $x\in G$, $X\in \mathfrak{g}$ and $f\in C^\infty(G)$, we define 
$$
Xf(x):=\frac{d}{dt} f(x\exp(tX))\bigg|_{t=0}.
$$
Notice that $X\overline{f}(x) = \overline{Xf(x)}$, for all $x\in G$, $X\in \mathfrak{g}$ and $f\in C^\infty(G)$.

The  symbol of a continuous linear operator $P: C^{\infty}(G) \to C^{\infty}(G)$ in $x\in G$ and $\phi \in \mbox{{Rep}}(G)$, $\phi=(\phi_{ij})_{i,j=1}^{d_\phi}$ is defined by 
\begin{equation}
	\label{symbol}
	\sigma_P(x,\phi) := \phi(x)^*(P\phi)(x) \in \mathbb{C}^{d_\phi \times d_\phi},
\end{equation}
where $(P\phi)(x)_{ij}:= (P\phi_{ij})(x)$, for all $1\leq i,j \leq d_\phi$, and we have the following quantization
$$
Pf(x) = \sum_{[\phi] \in \widehat{G}} \dim (\phi) \mbox{Tr} \left(\phi(x)\sigma_P(x,\phi)\widehat{f}(\phi)\right)
$$
of the operator $P$, for every $f \in C^\infty(G)$ and $x\in G$.
When $P: C^\infty(G) \to C^\infty(G)$ is a continuous linear left-invariant operator its symbol $\sigma_P$ is independent of $x\in G$ and
$$
\widehat{Pf}(\phi) = \sigma_P(\phi)\widehat{f}(\phi),
$$
for all $f \in C^\infty(G)$ and $[\phi] \in \widehat{G}$.

Let $X\in \mathfrak{g}$ be a vector field normalized by the norm induced by the Killing form. It is easy to see that the operator $iX$ is a left-invariant symmetric operator on $L^2(G)$. Thus, for all $[\phi] \in \widehat{G}$ we can choose a representative $\phi$ such that $\sigma_{iX}(\phi)$ is a diagonal matrix, with entries $\lambda_m \in \mathbb{R}$, $1 \leq m \leq d_\phi$. By the linearity of the symbol, we obtain
$$
\sigma_X(\phi)_{mn} = i\lambda_m \delta_{mn}, \quad \lambda_j \in \mathbb{R}.
$$
In this case we have by \eqref{symbol} that $\sigma_X(\overline{\phi}) = \overline{\sigma_X(\phi)}$,  that is, 
$$
\sigma_X(\overline{\phi})_{mn} = -i\lambda_m \delta_{mn}.
$$
Moreover, we have
$$
|\lambda_m(\phi)| \leq \jp{\phi},
$$
for all $[\phi] \in \widehat{G}$ and $1 \leq m \leq d_\phi$.


\section{Sufficient conditions for the global hypoellipticity}\label{suf-GH}
Let $G$ be a compact Lie group and $X \in \mathfrak{g}$ be a normalized vector field. In this section we deal with the global hypoellipticity of the operator $P:C^\infty(G) \to C^\infty(G)$ given by
\begin{equation*}\label{vekua2}
	Pu:= Xu- qu-p\bar{u},
\end{equation*}
where $p,q\in \mathbb{C}$, with $p \neq 0$. The case $p=0$ was studied in \cite{KMR20b}.
For $u \in \mathcal{D}'(G)$, we define its conjugate $\bar{u}$ by
$$
\jp{\bar{u}, \varphi} := \overline{ \jp{u,\bar{\varphi}}}, \quad \varphi \in C^\infty(G)
$$
and it is clear that $\bar{u} \in \mathcal{D}'(G)$. Hence, we can extend the operator $P$ to the space of distributions $\mathcal{D}'(G)$.
 
We recall that the operator $P$ is globally hypoelliptic if the conditions $u\in \mathcal{D}'(G)$ and $Pu\in C^\infty(G)$ imply  $u \in C^\infty(G)$.

Our goal is to obtain necessary and sufficient conditions for the global hypoellipticity of $P$ and for this we will analyze the behavior of the Fourier coefficients of $Pu$.

First, for each $[\xi] \in \widehat{G}$ we choose a representative such that 
$$
\sigma_X(\xi)_{mn} = i \lambda_m(\xi)\delta_{mn} \quad \text{and} \quad \sigma_X(\bar{\xi})_{mn} = -i \lambda_m(\xi)\delta_{mn},
$$
for $1 \leq m,n \leq d_\xi$, where $\delta_{mn}$ is the Kronecker's delta. This assumption means that if $[\xi] \not\sim [\overline{\xi}]$, then the representatives chosen for each one of this equivalence classes satisfy $\overline{\xi}(x) = \overline{\xi(x)}$, for all $x\in G$. Notice that with this choice we have $\lambda_m(\bar{\xi}) = -\lambda_m(\xi)$, for all $[\xi] \in \widehat{G}$, $1 \leq m \leq d_\xi$.

Hence, for $[\xi] \in \widehat{G}$ we obtain
\begin{align*}
	\widehat{Pu}(\xi)_{mn} &= \widehat{Xu}(\xi)_{mn} - q\widehat{u}(\xi)_{mn} - p\widehat{\bar{u}}(\xi)_{mn} \\
		&=\left[\sigma_X(\xi)\widehat{u}(\xi)\right]_{mn}- q\widehat{u}(\xi)_{mn} - p\widehat{\bar{u}}(\xi)_{mn}  \\
		&=(i\lambda_m(\xi)-q)\widehat{u}(\xi)_{mn} - p\widehat{\bar{u}}(\xi)_{mn} 
\end{align*}
Similarly, 
\begin{align*}
	\widehat{Pu}(\bar{\xi})_{mn} &= \widehat{Xu}(\bar{\xi})_{mn} - q\widehat{u}(\bar{\xi})_{mn} - p\widehat{\bar{u}}(\bar{\xi})_{mn} \\
	&=\left[\sigma_X(\bar{\xi})\widehat{u}(\bar{\xi})\right]_{mn}- q\widehat{u}(\bar{\xi})_{mn} - p\overline{\widehat{u}({\xi})_{mn}} \\
	&=(i\lambda_m(\bar{\xi})-q)\widehat{u}(\bar{\xi})_{mn} - p\overline{\widehat{u}({\xi})_{mn}} \\
	&=(-i\lambda_m({\xi})-q)\widehat{u}(\bar{\xi})_{mn} - p\overline{\widehat{u}({\xi})_{mn}}
\end{align*}
Now, using the fact that $\widehat{\bar{u}}(\xi) = \overline{\widehat{u}(\bar{\xi})}$, we obtain the following system of equations
\begin{equation}\label{system}
\left\{
\begin{array}{rrr}
		\widehat{Pu}(\xi)_{mn} =&(i\lambda_m(\xi)-q)\widehat{u}(\xi)_{mn}& - p\widehat{\bar{u}}(\xi)_{mn} \\
		\widehat{\overline{Pu}}(\xi)_{mn} =&- \bar{p}{\widehat{u}({\xi})_{mn}}&+(i\lambda_m({\xi})-\bar{q})\widehat{\bar{u}}(\xi)_{mn}
\end{array} 
\right.
\end{equation}
Denote
\begin{equation}\label{delta-def}
\Delta(\xi)_m:= -\lambda_m(\xi)^2+|q|^2-|p|^2-i2\lambda_m(\xi)\textrm{Re}(q).
\end{equation}
and by \eqref{system} we obtain
\begin{equation}\label{relation-sol}
	\Delta(\xi)_m \widehat{u}(\xi)_{mn} = (i\lambda_m(\xi)-\bar{q})\widehat{Pu}(\xi)_{mn}+p\widehat{\overline{Pu}}(\xi)_{mn}
\end{equation}
\begin{theorem}\label{thm-suf}
The operator $P:\mathcal{D}'(G)\to \mathcal{D}'(G)$ given by
$$
Pu:=Xu-qu-p\bar{u}
$$
 is globally hypoelliptic if one of the following statements holds
\begin{enumerate}
	\item $|p|>|q|$;
	\item $|p|<|q|$ and $\mathrm{Re}(q) \neq 0$;
	\item $\exists M>0$ such that
	$$
\jp{\xi} \geq M \implies 	\Big|\lambda_m(\xi)^2-(|q|^2-|p|^2)\Big| \geq \jp{\xi}^{-M},
	$$
	for all $1 \leq m \leq d_\xi$.
\end{enumerate}
\end{theorem}
\begin{proof}
	Let $u\in \mathcal{D}'(G)$ such that $Pu=f \in C^\infty(G)$. Let us prove that $u\in C^\infty(G)$ if one of the three statements above holds. Since $f \in C^\infty(G)$, for all $N>0$ there exists $C_N>0$ such that 
\begin{equation}\label{f-smooth}
	|\widehat{f}(\xi)_{mn}| \leq C_N\jp{\xi}^{-N}, \quad 1 \leq m,n \leq d_\xi, \ [\xi]\in \widehat{G}, 
\end{equation}
	and the same estimate is valid for $\overline{f}$. 
	
	Assume that 1. holds, that is, $|p|>|q|$. In this case, we have
	$$
	|\Delta(\xi)_m| \geq |\mathrm{Re}(\Delta(\xi)_m)| = |-\lambda_m(\xi)^2+|q|^2-|p|^2|\geq |p|^2-|q|^2>0.
	$$
	By \eqref{relation-sol} we obtain 
	\begin{align*}
		|\widehat{u}(\xi)_{mn} |&\leq \frac{1}{|\Delta(\xi)_m|}\left(|i\lambda_m(\xi)-\bar{q} ||\widehat{f}(\xi)_{mn}| + |p||\widehat{\overline{f}}(\xi)_{mn}|\right)\\
		&\leq \frac{1}{|p|^2-|q|^2}(|\lambda_m(\xi)||\widehat{f}(\xi)_{mn}|+|\bar{q}||\widehat{f}(\xi)_{mn}|+|p||\widehat{\overline{f}}(\xi)_{mn}|)\\
		&\leq \frac{1}{|p|^2-|q|^2}(\jp{\xi}|\widehat{f}(\xi)_{mn}|+|{q}||\widehat{f}(\xi)_{mn}|+|p||\widehat{\overline{f}}(\xi)_{mn}|).
	\end{align*}
Hence, given $N>0$ we obtain from \eqref{f-smooth}
$$
|\widehat{u}(\xi)_{mn} |\leq\left(\frac{C_{N+1}+C_N(|q|+|p|)}{|p|^2-|q|^2}\right) \jp{\xi}^{-N}
$$
and we conclude that $u\in C^\infty(G)$ by \eqref{cinfty}.

Assume now that 2. holds, that is, $|p|<|q|$ and $\mathrm{Re}(q) \neq 0$. If $|\lambda_m(\xi)|^2 \leq \frac{1}{2}(|q|^2-|p|^2)$ then
$$
|\Delta(\xi)_m|^2 \geq |\mathrm{Re}(\Delta(\xi)_m)|^2 = (-\lambda_m(\xi)^2+|q|^2-|p|^2)^2 \geq \frac{1}{4}(|q|^2-|p|^2)^2 >0.
$$
On the other hand, if $|\lambda_m(\xi)|^2 > \frac{1}{2}(|q|^2-|p|^2)$ then
$$
|\Delta(\xi)_m|^2 \geq |\mathrm{Im}(\Delta(\xi)_m)|^2 = 4\mathrm{Re}(q)^2\lambda_m(\xi)^2 \geq 2\mathrm{Re}(q)^2(|q|^2-|p|^2) >0.
$$
Thus we  obtain $C>0$ such that
$$
|\Delta(\xi)_m| \geq C,
$$
for all $[\xi]\in \widehat{G}$, and $1 \leq m \leq d_\xi$. Analogously to the previous case we conclude that $u \in C^\infty(G)$.

Finally, assume that 3. holds, so there exists $M>0$ such that 	$$\Big|\lambda_m(\xi)^2-(|q|^2-|p|^2)\Big| \geq \jp{\xi}^{-M},$$
for all $1 \leq m \leq d_\xi$, whenever $\jp{\xi}\geq M$. Hence, for $\jp{\xi}\geq M$ we have
$$
|\Delta(\xi)_m| \geq |\mathrm{Re}(\Delta(\xi)_m)| = |-\lambda_m(\xi)^2+|q|^2-|p|^2| \geq \jp{\xi}^{-M}.
$$
Following the same steps as the first case, we obtain
$$
	|\widehat{u}(\xi)_{mn} | \leq \jp{\xi}^M(\jp{\xi}|\widehat{f}(\xi)_{mn}|+|{q}||\widehat{f}(\xi)_{mn}|+|p||\widehat{\overline{f}}(\xi)_{mn}|).
	$$
	Using the fact that $f\in C^{\infty}(G)$, for $N>0$ we have
	$$
		|\widehat{u}(\xi)_{mn} | \leq(C_{N+M+1} + C_{N+M}(|q|+|p|))\jp{\xi}^{-N},
	$$
	for all $1 \leq m,n \leq d_\xi$, whenever $\jp{\xi} \geq M$. Since the set
	$$
	\{[\xi] \in \widehat{G}; \jp{\xi} <M \}
	$$
	is finite, we can obtain $C'_N>0$ for each $N>0$, such that 
	$$
		|\widehat{u}(\xi)_{mn} | \leq C'_N\jp{\xi}^{-N},
	$$
		for all $[\xi]\in \widehat{G}$, and $1 \leq m,n \leq d_\xi$. Therefore $u\in C^\infty(G)$ and the theorem is proved.
\end{proof}
\begin{example}
Let $G$ be a compact Lie group and let $X \in \mathfrak{g}$. Consider the operator $P_1$ defined by
$$
P_1u:= Xu-iu-2\overline{u},
$$
that is, $p=2$ and $q=i$.  From the condition 1 of Theorem \ref{thm-suf} the operator $P_1$ is globally hypoelliptic. Consider now the operator $P_2$ given by
$$
P_2u:= Xu-6u-(3+4i)\overline{u},
$$
that is, $p=3+4i$ and $q=6$. From the condition 2 of Theorem \ref{thm-suf} the operator $P_2$ is globally hypoelliptic. 
\end{example}
To give an example where we need to apply the condition 3 of Theorem \ref{thm-suf} we need to know the behavior of the symbol of $X$. For the next example, we will follow the notation and results from \cite{RT13} and Chapter 11 of \cite{livropseudo} to analyze an operator defined on $\SU$. 
\begin{example}\label{ex-su2}
Let $G=\SU$ and let $\widehat{\textrm{SU}(2)}$ be the unitary dual of $\textrm{SU}(2)$, that is, $\widehat{\textrm{SU}(2)}$ consists of equivalence classes $[\textsf{t}^\ell]$ of continuous irreducible unitary representations
$\textsf{t}^\ell:\textrm{SU}(2)\to \mathbb{C}^{(2\ell+1)\times (2\ell+1)}$, $\ell\in\frac12\mathbb{N}_{0}$,
of matrix-valued functions satisfying
$\textsf{t}^\ell(xy)=\textsf{t}^\ell(x)\textsf{t}^\ell(y)$ and $\textsf{t}^\ell(x)^{*}=\textsf{t}^\ell(x)^{-1}$ for all
$x,y\in\textrm{SU}(2)$.
We will use the standard convention of enumerating the matrix elements
$\textsf{t}^\ell_{mn}$ of $\textsf{t}^\ell$ using indices $m,n$ ranging between
$-\ell$ to $\ell$ with step one, i.e. we have $-\ell\leq m,n\leq\ell$
with $\ell-m, \ell-n\in\mathbb{N}_{0}.$ For $\ell \in \tfrac{1}{2} \mathbb{N}_0$ we have
$$
\jp{\ell} := \jp{\textsf{t}^\ell} = \sqrt{1+\ell(\ell+1)}. 
$$

Let $X=\partial_0$, where $\partial_0$ is the neutral operator. There is no loss of generality assuming that the vector field is $\partial_0$ because all vector fields define on $\textrm{SU}(2)$ can be conjugated to $\partial_0$  and this conjugation does not affect lower order terms. We have that
$
\sigma_{\partial_0}(\ell)_{mn}=im\delta_{mn},
$
for all $\ell\in \frac{1}{2} \mathbb{N}_0$, that is, 
$$\lambda_m(\ell) = m, $$ for all $\ell \in \frac{1}{2}\mathbb{N}_0$, $-\ell \leq m \leq \ell$ with $\ell -m \in \mathbb{N}_0$.

Consider the operator $P_3$ given by
$$
P_3u:= \partial_0 u - iu - \frac{1}{\sqrt{3}} \overline{u},
$$
that is, $p=\frac{1}{\sqrt{3}}$ and $q=i$. Notice that
$$
\left|m^2-\left(1-\frac{1}{3}\right)\right|  = \left|m^2-\frac{2}{3}\right| > \frac{1}{3},
$$
for all $m \in \frac{1}{2} \mathbb{Z}$. Therefore, the operator $P_3$ is globally hypoelliptic by the condition 3 of Theorem \ref{thm-suf}.
\end{example}

\section{Necessary conditions for the global hypoellipticity}\label{nec-GH}
The Theorem \ref{thm-suf} gives us sufficient conditions for the global hypoellipticity of the operator $P$. Our next result says that these conditions are also necessary for the global hypoellipticity of $P$ in the case where we have
$[\xi] \neq [\bar{\xi}],$
for all non-trivial $[\xi] \in \widehat{G}$. The next proposition gives us the first necessary condition for the global hypoellipticity of $G$.
\begin{proposition}\label{lemma-nec}
Assume that $[\xi] \neq [\bar{\xi}],$
for all non-trivial $[\xi] \in \widehat{G}$.	If $\Delta(\xi)_m=0$ for infinitely many $[\xi]\in \widehat{G}$, $1 \leq m \leq d_\xi$, then $P$ is not globally hypoelliptic.
\end{proposition}
\begin{proof}
	Let $\{ [\xi_j]\}_{j\in\mathbb{N}}$ be a sequence in $\widehat{G}$ such that $\Delta(\xi_j)_{m_j}=0$, for some $1 \leq m_j \leq d_{\xi_j}$. We may assume that $\xi_j \neq \overline{\xi_k}$, for any $j,k \in \mathbb{N}$.
	Since
	$$
	\Delta(\xi)_m= -\lambda_m(\xi)^2+|q|^2-|p|^2-i2\lambda_m(\xi)\textrm{Re}(q),
	$$
and $\lambda_m(\bar{\xi}) = -\lambda_m(\xi)$, we have
$$
\Delta(\bar{\xi})_m = \overline{\Delta(\xi)_m},
$$
	for all $[\xi] \in \widehat{G}$, $ 1 \leq m \leq d_\xi$. In particular, $\Delta(\bar{\xi_j})_{m_j}=0$, for all $j\in \mathbb{N}$. Define
	$$
		\widehat{u}(\xi)_{mn} = 
		\left\{
\begin{array}{ll}
	i\lambda_{m_j}(\xi_j)-\bar{q}, & \text{if } [\xi]=[\xi_j], \ m=m_j\\
	p, & \text{if } [\xi]=[\overline{\xi_j}], \ m=m_j\\
	0, & \text{otherwise}.
\end{array}
	\right.
	$$
	Notice that $\widehat{u}(\xi)_{mn}$ is well--defined because our assumption $[\xi] \neq [\bar{\xi}]$, for all non-trivial representation. Let us prove that $\widehat{u}(\xi)_{mn}$ is the Fourier coefficient of a distribution $u\in\mathcal{D}'(G)$ that is not a smooth function. Indeed, we have
	$$
	|\widehat{u}(\xi_j)_{mn}| \leq |i\lambda_m(\xi_j)-\bar{q}| \leq |\lambda_m(\xi_j|)| + |q| \leq (1+|q|)\jp{\xi_j},
	$$
	for all $1 \leq m,n \leq d_{\xi_j}$, and
	$$
		|\widehat{u}(\bar{\xi_j})_{mn}| \leq |p| \leq |p| \jp{\bar{\xi_j}},
	$$
		for all $1 \leq m,n \leq d_{\bar{\xi_j}}$.	Hence for $C=\max\{1+|q|, |p|\}$, we have
	$$
	|\widehat{u}(\xi)_{mn}| \leq C \jp{\xi},
	$$
	for all $[\xi] \in \widehat{G}$, $1 \leq m,n \leq d_\xi$, which implies that $u \in \mathcal{D}'(G)$ by \eqref{distr}. Notice that
	$$
	|\widehat{u}(\bar{\xi_j})_{m_jn}| = |p| \neq 0,
	$$
	for all $j \in \mathbb{N}$, so we conclude that $u \in \mathcal{D}'(G) \setminus C^\infty(G)$ because \eqref{cinfty} does not hold.
	
	Let us show now that $Pu=0$. We know that
	$$
	\widehat{Pu}(\xi)_{mn} = (i\lambda_m(\xi)-q)\widehat{u}(\xi)_{mn} - p \overline{\widehat{u}(\bar{\xi})_{mn}}.
	$$
It is enough to show that $\widehat{Pu}(\xi_j)_{m_jn} = \widehat{Pu}(\bar{\xi_j})_{m_jn}=0$, for all $j \in \mathbb{N}$. Indeed, we have
	\begin{align*}
		\widehat{Pu}(\xi_j)_{m_jn} & = (i\lambda_{m_j}(\xi_j)-q)\widehat{u}(\xi_j)_{mn} - p \overline{\widehat{u}(\bar{\xi_j})_{mn}} \\
		& = (i\lambda_{m_j}(\xi_j)-q) (i\lambda_{m_j}(\xi_j) -\bar{q})- p \bar{p}\\
		&=-\lambda_{m_j}(\xi_j)^2 +|q|^2-|p|^2 - i2\mathrm{Re}(q)\lambda_{m_j}(\xi_j)\\
		& = \Delta(\xi_j)_{m_j} \\
		&= 0.
	\end{align*}
Moreover,
\begin{align*}
	\widehat{Pu}(\bar{\xi_j})_{m_jn} & = (i\lambda_{m_j}(\bar{\xi_j})-q)\widehat{u}(\bar{\xi_j})_{mn} - p  \overline{\widehat{u}({\xi_j})_{mn}}\\
	& = (i\lambda_{m_j}(\bar{\xi_j})-q) p- p \overline{(i\lambda_{m_j}({\xi_j}) -\bar{q})} \\
	& = (-i\lambda_{m_j}({\xi_j})-q) p- p {(-i\lambda_{m_j}({\xi_j}) -{q})} \\
	&= 0.
\end{align*}
Therefore  $Pu=0$ and we conclude that $P$ is not globally hypoelliptic.
\end{proof}
Notice that to construct a singular solution for the equation $Pu=0$ in the proof of the last proposition we define the Fourier coefficient of $u$ at $[\xi_j]$ and $[\overline{\xi_j}]$ and this was possible because our assumption that $[\xi] \neq [\overline{\xi}]$, for all $[\xi] \in \widehat{G}$. If $[\eta] = [\overline{\eta}]$ for some $[\eta]\in \widehat{G}$, then there exists a unitary matrix $A$ such that $A\eta(x)=\overline{\eta(x)}A$, for all $x\in G$, witch implies that 
$A\widehat{u}(\eta)=\widehat{u}(\overline{\eta})A$. In this case, we could not define indiscriminately the values of $\widehat{u}$ at $\eta$ and at $\overline{\eta}$. In Section \ref{operatoronsu2} we will present this result for a group that does not satisfy this hypothesis.
\begin{theorem}\label{gh-nec}
	Assume that $[\xi] \neq [\bar{\xi}],$
	for all non-trivial $[\xi] \in \widehat{G}$. If $P$ is globally hypoelliptic, then one of the following conditions holds:
	\begin{enumerate}
		\item $|p|>|q|$;
		\item $|p|<|q|$ and $\mathrm{Re}(q) \neq 0$;
		\item $\exists M>0$ such that
		$$
		\jp{\xi} \geq M \implies 	\Big|\lambda_m(\xi)^2-(|q|^2-|p|^2)\Big| \geq \jp{\xi}^{-M},
		$$
		for all $1 \leq m \leq d_\xi$.
	\end{enumerate}
\end{theorem}
\begin{proof}
	Let us prove by contradiction. If none of the assumptions are satisfied, then we have two possibilities:
	\begin{enumerate}[(a)]
		\item $|p|<|q|, \mathrm{Re}(q) =0$, and 3. does not hold;
		\item $|p|=|q|$ and 3. does not hold.
	\end{enumerate}

If (a) holds, then
$$
\Delta(\xi)_m=-\lambda_m(\xi)^2+|q|^2-|p|^2,
$$
and for every $k \in \mathbb{N}$, there exists $[\xi_k] \in \widehat{G}$ such that $\jp{\xi_k} \geq k$ and 
$$
|\Delta(\xi_k)_{m_k}| = |-\lambda_{m_k}(\xi_k)^2+|q|^2-|p|^2| \leq \jp{\xi_k}^{-k},
$$
for some $1 \leq m_k \leq d_{\xi_k}$. Since $\Delta(\bar{\xi_k})_{m_k} = \overline{\Delta(\xi_k)_{m_k}}$ and $\jp{\xi_k} = \jp{\bar{\xi_k}}$, we also have
$$
|\Delta(\bar{\xi_k})_{m_k}|  \leq \jp{\xi_k}^{-k}.
$$

If (b) holds, we have
$$
\Delta(\xi)_m=-\lambda_m(\xi)^2+i2\mathrm{Re}(q)\lambda_m(\xi),
$$
and for every $\ell \in \mathbb{N}$, there exists $[\xi_{\ell}] \in \widehat{G}$ such that $\jp{\xi_{\ell}} \geq \ell$ and 
\begin{equation}\label{lambda-estimate}
|\lambda_{m_\ell}(\xi_{\ell})| \leq \jp{\xi_{\ell}}^{-\ell},
\end{equation}
for some $1 \leq m_\ell \leq d_{\xi_{\ell}}$. In particular we have $|\lambda_{m_\ell}(\xi_{\ell})| \leq 1$, for all $\ell \in \mathbb{N}$. Notice that
$$
|\Delta(\xi_\ell)_{m_\ell}|=|\lambda_{m_\ell}(\xi_\ell)|(|\lambda_{m_\ell}(\xi_\ell)| + 2|\mathrm{Re}(q)|) \leq (1+2|\mathrm{Re}(q)|)\jp{\xi_{\ell}}^{-\ell}
$$
for all $\ell \in \mathbb{N}$, and again the same is true for the dual representation $\bar{\xi_{\ell}}$. 

Hence, in both cases there exist a sequence $\{[\xi_k] \}_{k \in \mathbb{N}}$ in $\widehat{G}$ and $C>0$ such that $\jp{\xi_k} \geq k$ and 
\begin{equation}\label{estimate-f}
|\Delta(\xi_k)_{m_k}| \leq C \jp{\xi_k}^{-k}, \quad k \in \mathbb{N},
\end{equation}
for some $ 1 \leq m_k \leq d_{\xi_k}$.
From now we will treat both cases altogether. Moreover, since $P$ is globally hypoelliptic we may assume, by Proposition \ref{lemma-nec}, that there exists $L>0$ such that $\Delta(\xi_k)_{m_k} \neq 0$, for all $k\geq L$. Thus
$$
0 < |\Delta(\xi_k)_{m_k}| \leq C\jp{\xi_k}^{-k}, \quad k \geq L.
$$
Define
	$$
\widehat{f}(\xi)_{mn} = 
\left\{
\begin{array}{ll}
\Delta(\xi)_m & \text{if } [\xi]=[\xi_k] \text{ or }  [\xi]=[\bar{\xi_k}], \ m=m_k\\
	0, & \text{otherwise}.
\end{array}
\right.
$$
By \eqref{estimate-f} we have that $\{\widehat{f}(\xi)_{mn}\}$ are the Fourier coefficients of a function $f \in C^\infty(G)$. Let us construct now a singular solution for $Pu=f$. Define
	$$
\widehat{u}(\xi)_{mn} = 
\left\{
\begin{array}{ll}
	i\lambda_m(\xi)-\bar{q}+p & \text{if } [\xi]=[\xi_k] \text{ or }  [\xi]=[\bar{\xi_k}], \ m=m_k\\
	0, & \text{otherwise}.
\end{array}
\right.
$$
Notice that
$$
|\widehat{u}(\xi_k)_{m_kn}| \leq |\lambda_{m_k}(\xi_k)|+|p-\bar{q}| \leq C \jp{\xi_k},
$$
for all $k \in \mathbb{N}$, which implies that $u \in \mathcal{D}'(G)$. Moreover,
\begin{align*} 
	\widehat{Pu}(\xi_k)_{m_kn} &=(i\lambda_{m_k}(\xi_k)-q)\widehat{u}(\xi_k)_{m_kn} - p\overline{\widehat{{u}}(\bar{\xi_k})_{m_kn}} \\
	&=(i\lambda_{m_k}(\xi_k)-q)(i\lambda_{m_k}(\xi_k)-\bar{q}+p)-p\overline{(	i\lambda_{m_k}(\bar{\xi_k})-\bar{q}+p)}\\
	&=\Delta(\xi_k)_{m_k}\\
	&=\widehat{f}(\xi_k)_{m_kn}.
\end{align*}
Therefore $\widehat{Pu}(\xi)_{mn} = \widehat{f}(\xi)_{mn}$, for all $[\xi]\in \widehat{G}$, $1 \leq m,n \leq d_\xi$. It remains to show that $u \notin C^\infty(G)$. There are three situations to be consider:

\noindent {\bf(I)} $\mathrm{Re}(p-\bar{q}) \neq 0$.

Notice that
$$
|\widehat{u}(\xi_k)_{m_kn}| \geq |\mathrm{Re}(\widehat{u}(\xi_k)_{m_kn})| = |\mathrm{Re}(p-\bar{q})|>0,
$$
for all $k \in \mathbb{N}$, which implies that $u \notin C^\infty(G)$.

\noindent {\bf(II)} $\mathrm{Re}(p-\bar{q}) = 0$ and $\mathrm{Im}(p-\bar{q}) \neq 0$.

Since $\lambda_{m_k}(\bar{\xi_k}) = -\lambda_{m_k}(\xi_k)$, we may assume without loss of generality that
$$
\mathrm{sgn} (\lambda_{m_k}(\xi_k)) = \mathrm{sgn}(\mathrm{Im}(p-\bar{q})),
$$
when $\lambda_{m_k}(\xi_k) \neq 0$.
Hence,
$$
|\widehat{u}(\xi_k)_{m_kn}| \geq |\mathrm{Im}(\widehat{u}(\xi_k)_{m_kn})| = |\lambda_{m_k}(\xi_k)+\mathrm{Im}(p-\bar{q})|>|\mathrm{Im}(p-\bar{q})|>0,
$$
for all $k \in \mathbb{N}$, and again we conclude that $u \notin C^\infty(G)$.

\noindent {\bf(III)} $p=\bar{q}$.

For this case, let us construct a singular solution for $Pu=if$.  Define  
	$$
\widehat{u}(\xi)_{mn} = 
\left\{
\begin{array}{ll}
	-\lambda_m(\xi)-i2p & \text{if } [\xi]=[\xi_k] \text{ or }  [\xi]=[\bar{\xi_k}], \ m=m_k\\
	0, & \text{otherwise}.
\end{array}
\right.
$$
As before we have $u \in \mathcal{D}'(G)$. If $\mathrm{Re}(p) \neq 0$, then
$$
|\widehat{u}(\xi_k)_{m_kn}| \geq |\mathrm{Im}(\widehat{u}{(\xi_k)}_{m_kn})| = 2|\mathrm{Re}(p)| > 0,
$$
that is, $u \notin C^\infty(G)$. If $\mathrm{Re}(p)=0$, then $\mathrm{Im}(p)\neq 0$, because $p \neq 0$. So,
$$
|\widehat{u}({\xi_k})_{m_kn}| = |\lambda_{m_k}(\xi_k)-2\mathrm{Im}(p)|> |\mathrm{Im}(p)| >0,
$$
for $k$ large enough, where the last inequality comes from \eqref{lambda-estimate}. Hence $u \notin C^\infty(G)$. Finally,
\begin{align*} 
	\widehat{Pu}(\xi_k)_{m_kn} &=(i\lambda_{m_k}(\xi_k)-q)\widehat{u}(\xi_k)_{m_kn} - p\overline{\widehat{{u}}(\bar{\xi_k})_{m_kn}} \\
	&=(i\lambda_{m_k}(\xi_k)-\bar{p})(-\lambda_{m_k}(\xi_k)-i2p)-p(\overline{-\lambda_{m_k}(\bar{\xi_k})-i2p})\\
	&=-i\lambda_{m_k}(\xi_k)^2+\mathrm{Re}(p)\lambda_{m_k}(\xi_k)\\
	&=i\Delta(\xi_k)_{m_k}\\
	&=i\widehat{f}(\xi_k)_{m_kn}.
\end{align*}
So, $\widehat{Pu}(\xi)_{mn}=i\widehat{f}(\xi)_{mn}$, for all $[\xi] \in \widehat{G}$, $1 \leq m,n \leq d_{\xi}$, hence $Pu=if$, which is a contradiction because $P$ is globally hypoelliptic. 
\end{proof}

In order to apply Theorem \ref{gh-nec} to give an example of an operator $P$ that is not globally hypoelliptic we need to verify if the compact Lie group $G$ satisfies the hypothesis $[\xi] \neq [\bar{\xi}],$
for all non-trivial $[\xi] \in \widehat{G}$. Since this assumption holds for $\mathbb{T}^d$, for any $d\in \mathbb{N}$,  we conclude that the hypothesis is verified for $\mathbb{T}^d \times G$, for any $d\in \mathbb{N}$ and any  compact Lie group $G$. We refer  \cite{BD95} and  \cite{KMR20}  for a detail discussion about representations of a product of compact Lie groups.
\begin{example}\label{notgh}
Consider the compact Lie group $G=\mathbb{T}^1 \times {\SU}$ and let $P$ the operator defined by
$$
Pu(t,x):= \partial_tu(t,x) + \partial_0u(t,x) - \tfrac{\sqrt{5}}{2}iu(t,x) - \overline{u(t,x)}, \quad (t,x)\in \mathbb{T}^1\times {\SU},
$$
where $\partial_0$ is the neutral operator on ${\SU}$ as in Example \ref{ex-su2}. We have $\widehat{G} \sim \mathbb{Z} \times \frac{1}{2}\mathbb{N}_0$ and the symbol of $X:=\partial_t+\partial_0$ is given by
$$
\sigma_X(\tau, \ell)_{mn} = i(\tau+m) \delta_{mn},
$$
for all $\tau \in \mathbb{Z}$, $\ell \in \frac{1}{2}\mathbb{N}_0$, $-\ell \leq m,n \leq \ell$ with $\ell-m, \ell-n \in \mathbb{Z}$. Clearly conditions 1 and 2 of Theorem \ref{gh-nec} do not hold. Let us see that condition 3 also fails. Indeed, we have
$$
\left|(\tau+m)^2-\left(\frac{5}{4}-1\right)\right|=\left|(\tau+m)^2-\frac{1}{4}\right|=0
$$
for $\tau=0$ and $m=\frac{1}{2}$, so condition 3 does not hold because $m=\frac{1}{2}$ appears for all $\ell \in \frac{1}{2}\mathbb{N}_0 \setminus \mathbb{N}_0$. Therefore the operator $P$ is not globally hypoelliptic.
\end{example}

\begin{corollary}
	Assume that $[\xi] \neq [\bar{\xi}],$
for all non-trivial $[\xi] \in \widehat{G}$, and $|p|=|q|$. If $P$ is globally hypoelliptic, then $G$ is isomorphic to a torus.
\end{corollary}
\begin{proof}
	Since $P$ is globally hypoelliptic, we are in the situation 3. of the Theorem \ref{gh-nec}. Hence, using the fact the $|p|=|q|$, there exists $M>0$ such that 
	$$
	\jp{\xi} \geq M \implies 	|\lambda_m(\xi)| \geq \jp{\xi}^{-M},
	$$
	for all $1 \leq m \leq d_\xi$. This property implies that the vector field $X$ is globally hypoelliptic (Theorem 3.3 of \cite{KMR20b}). Therefore $G$ must be isomorphic to a torus, because the Greenfield-Wallach conjecture is valid for compact Lie groups, proved by Greenfield and Wallach in \cite{GW73a}. 
	
\end{proof}

We point out that when $G=\mathbb{T}^2$ the results obtained in this section extend the results in \cite{BDM14} about the global hypoellipticity of the operator
$$
Pu=\partial_t u+ c \partial_x u-qu - p\bar{u},
$$
when $c \in \mathbb{R}$ because $\partial_t+c\partial_x \in \widehat{\mathbb{T}^2}$ and $\mathbb{T}^2$ satisfies the hypothesis in Theorem \ref{gh-nec}. 

Although the hypothesis $\mathrm{Im} (c) \neq 0$ is a forth condition in Theorems \ref{thm-suf} and \ref{gh-nec} on $\mathbb{T}^2$ (see Theorem 1, \cite{BDM14}), this does not hold in general for compact Lie groups. For instance, consider the operator
$$
Pu=\partial_tu+c\partial_0u - qu - p\bar{u}
$$
defined in $G=\mathbb{T}^2\times \mathrm{SU}(2)$, where $\partial_0$ is the neutral operator on $\mathrm{SU}(2)$, with $c,p,q \in \mathbb{C}$ satisfying $\mathrm{Im}(c) \neq 0$ and $|p|=|q|$. The elements on $\widehat{G}$ can be identified with $\mathbb{Z} \times \tfrac{1}{2}\mathbb{N}_0$ and we have
$$
\sigma_{\partial_0} (\ell)_{mn} = im\delta_{mn},
$$
for all $\ell \in \tfrac{1}{2}\mathbb{N}_0$, $-\ell \leq m,n \leq \ell$, $\ell-m,\ell-n \in \mathbb{Z}$.

Since in $\mathbb{T}^2$ we have the property $[\xi] \neq [\bar{\xi}]$, for all non-trivial representations $[\xi] \in \widehat{\mathbb{T}^2}$, this property also holds for the product $\mathbb{T}^2 \times \textrm{SU}(2)$. In this case we can reply the same argument before the statement of Theorem \ref{thm-suf} to obtain
$$
\Delta(\tau,\ell)_m = -|\tau+cm|^2+|q|^2-|p|^2-i2\mathrm{Re}(q(\tau+\bar{c}m)),
$$
for all $\tau \in \mathbb{Z}$, $\ell \in \tfrac{1}{2}\mathbb{N}_0$, $-\ell \leq m \leq \ell$, $\ell-m\in \mathbb{Z}$. Moreover, we can easily adapt Proposition \ref{lemma-nec} to this operator. Hence, 
$$
\Delta(0,\ell)_0=0,
$$
for all $\ell \in \mathbb{N}$, that is, $\Delta(\tau,\ell)_m=0$ for infinitely many representations, which implies that $P$ is not globally hypoelliptic by Proposition \ref{lemma-nec}. 

We also can construct a counter-example that satisfies the condition 2. on Theorem \ref{thm-suf}. Consider the operator
$$
P_ku=\partial_tu+ic\partial_0u-qu-p\bar{u},
$$
with $c,q \in \mathbb{R}\setminus\{0\}$, and $p \in \mathbb{C}$ satisfying $|q|^2-|p|^2=|c|^2k^2$, with $k\in \mathbb{N}$. Notice that $|p|<|q|$, $\mathrm{Re}(q)\neq 0$, and
$$
\Delta(\tau,\ell)_m=-|\tau+ic m|^2+|q|^2-|p|^2-i2q\tau,
$$
for all $\tau \in \mathbb{Z}$, $\ell \in \tfrac{1}{2}\mathbb{N}_0$, $-\ell \leq m \leq \ell$, $\ell-m\in \mathbb{Z}$. Hence,
$$
\Delta(0,k)_k=-|c|^2k^2+|q|^2-|p|^2=0,
$$
for all $k\in \mathbb{N}$, which implies that $P_k$ is not globally hypoelliptic by Proposition \ref{lemma-nec}.
\section{Global Solvability}\label{sectionGS}
A natural question that appears on the study of properties of an operator is the existence of solutions, that is, given $f \in \mathcal{D}'(G)$, do there exist $u \in \mathcal{D}'(G)$ such that $Pu=f$?

For the operator $P$ defined on \eqref{operatorP} there are some compatibility conditions about which $f$ we can expect to solve the equation $Pu=f$. Precisely, if $Pu=f$, then from \eqref{relation-sol} we have
\begin{equation}\label{relation-sol2}
	\Delta(\xi)_m \widehat{u}(\xi)_{mn} = (i\lambda_m(\xi)-\bar{q})\widehat{f}(\xi)_{mn}+p\overline{\widehat{{f}}(\bar{\xi})_{mn}},
\end{equation}
for all $[\xi] \in \widehat{G}$, $1 \leq m,n \leq d_\xi$. Therefore, if $\Delta(\xi)_m=0$, the distribution $f$ must satisfy
\begin{equation}\label{compatibility}
(i\lambda_m(\xi)-\bar{q})\widehat{f}(\xi)_{mn}+p\overline{\widehat{{f}}(\bar{\xi})_{mn}}=0,
\end{equation}
for every $1 \leq n \leq d_\xi$. Consider $\mathbb{E}$ the subspace of $\mathcal{D}'(G)$ of distributions that have this property, that is,
$$
\mathbb{E}:= \{f \in \mathcal{D}'(G) ; \ \Delta(\xi)_m=0 \implies (i\lambda_m(\xi)-\bar{q})\widehat{f}(\xi)_{mn}+p\overline{\widehat{{f}}(\bar{\xi})_{mn}}=0, \ 1\leq n \leq d_\xi \}.
$$

We call the elements of $\mathbb{E}$ admissible distribution for the operator $P$ and we say that $P$ is globally solvable if for every $f \in \mathbb{E}$ there exists $u\in \mathcal{D}'(G)$ such that $Pu=f$, that is, $P(\mathcal{D}'(G)) = \mathbb{E}$.

This notion of solvability is different from the one studied on \cite{BDM14}, where the authors considered solvability on spaces of finite codimension. For us, global solvability means to solve the equation $Pu=f$, for every $f$ which this equation makes sense in terms of the compatibility condition \eqref{compatibility}.

Given $f \in \mathbb{E}$, our  goal is to determine $u\in \mathcal{D}'(G)$ to obtain $Pu=f$. By \eqref{relation-sol2}, when $\Delta(\xi)_m \neq 0$ we have
\begin{equation}\label{solutiondeltaneq0}
	 \widehat{u}(\xi)_{mn} =\frac{1}{	\Delta(\xi)_m} \left( (i\lambda_m(\xi)-\bar{q})\widehat{f}(\xi)_{mn}+p\widehat{\overline{f}}(\xi)_{mn}\right) 
\end{equation}

When $\Delta(\xi)_m = 0$ we are led to look for the conjugate representation and for this reason we will assume the condition $[\xi] \neq [\bar{\xi}],$
for all non-trivial $[\xi] \in \widehat{G}$. Notice that if $\Delta(\xi)_m=0$, then $\Delta(\bar{\xi})_m=0$. Thus, for each pair $([\xi], [\bar{\xi}])$ define 
\begin{equation}\label{solutiondeltaeq0}
\left\{ 
\begin{aligned}
	\widehat{u}(\xi)_{mn} &=0 \\
	\widehat{u}(\bar{\xi})_{mn} & = -\frac{\overline{\widehat{f}(\xi)_{mn}}}{\bar{p}},
\end{aligned}
\right.
\end{equation}
for every $1 \leq n \leq d_\xi$. Let us check that $\widehat{Pu}(\xi)_{mn} = \widehat{f}(\xi)_{mn}$. Indeed, by \eqref{system} we have
\begin{align*}
		\widehat{Pu}(\xi)_{mn} &=(i\lambda_m(\xi)-q)\widehat{u}(\xi)_{mn} - p\overline{\widehat{{u}}(\bar{\xi})_{mn}}\\
		&=-p \left(\overline{ -\frac{\overline{\widehat{f}(\xi)_{mn}}}{\bar{p}}}\right)\\
		&=\widehat{f}(\xi)_{mn}
\end{align*}
and
\begin{align*}
		\widehat{Pu}(\bar{\xi})_{mn} &= (-i\lambda_m({\xi})-q)\widehat{u}(\bar{\xi})_{mn} - p\overline{\widehat{u}({\xi})_{mn}}\\
		&=(-i\lambda_m({\xi})-q) \left(\overline{ -\frac{\overline{\widehat{f}(\xi)_{mn}}}{\bar{p}}}\right) \\
		&= \overline{\left({\frac{-i\lambda_m({\xi})+\bar{q}}{p}} \right)\widehat{f}(\xi)_{mn}}
\end{align*}
Since $f \in \mathbb{E}$, we have 
$$
\left({\frac{-i\lambda_m({\xi})+\bar{q}}{p}} \right)\widehat{f}(\xi)_{mn} = \overline{\widehat{f}(\bar{\xi})_{mn}},
$$
which implies that 
$$
\widehat{Pu}(\bar{\xi})_{mn}  = \overline{\overline{\widehat{f}(\bar{\xi})_{mn}}} = \widehat{f}(\bar{\xi})_{mn},
$$
for all $1 \leq n \leq d_\xi$. We point out that the solution given by \eqref{solutiondeltaeq0} is not unique.

Notice that when $\Delta(\xi)_m=0$, the decay of $\widehat{u}(\xi)_{mn}$ and $\widehat{f}(\xi)_{mn}$ are the same, then in order to guarantee that $u \in \mathcal{D}'(G)$ we need to estimate $\Delta(\xi)_m^{-1}$, when $\Delta(\xi)_m \neq 0$.

\begin{theorem}\label{GS}
Assume that  $[\xi] \neq [\bar{\xi}],$
for all non-trivial $[\xi] \in \widehat{G}$. Then $P$ is globally solvable if and only if one of the following conditions holds:
\begin{enumerate}
	\item $|p|>|q|$;
	\item $|p|<|q|$ and $\mathrm{Re}(q) \neq 0$;
	\item $\exists M>0$ such that
	$$
	\Big|\lambda_m(\xi)^2-(|q|^2-|p|^2)\Big| \geq \jp{\xi}^{-M},
	$$
	for all $1 \leq m \leq d_\xi$, whenever $\lambda_m(\xi)^2-(|q|^2-|p|^2)\neq 0$
\end{enumerate}
\end{theorem}
\begin{proof}
	By the previous discussion, given $f \in \mathbb{E}$ we need to proof that the sequence of Fourier coefficients given by \eqref{solutiondeltaneq0} and \eqref{solutiondeltaeq0} defines a distribution $u \in \mathcal{D}'(G)$. 
The proof of the sufficiency has the same estimates from the proof of Theorem  \ref{suf-GH} and will be omitted. 

To prove the necessity, assume that none of the conditions 1-3 are satisfied. Proceeding analogously to the proof of Theorem \ref{gh-nec}, we can obtain a sequence $\{[\xi_k]\}_{k \in \mathbb{N}}$ such that
\begin{equation}\label{counter}
0 < |\Delta(\xi_k)_{m_k}| \leq C \jp{\xi_k}^{-k},
\end{equation}
for some constant $C>0$ and $1 \leq m_k \leq d_{\xi_k}$, for all $k\in \mathbb{N}$. We may assume that $\xi_k \neq \bar{\xi_{\ell}}$, for every $k, \ell \in \mathbb{N}$.
Let us construct an admissible distribution $f \in \mathbb{E}$ such that there is no $u \in \mathcal{D}'(G)$ satisfying $Pu=f$. Define
$$
\widehat{f}(\xi)_{mn} = 
\left\{
\begin{aligned}
	{\overline{p}}^{-1}, & \quad  \text{if } \xi = \bar{\xi_k} \text{ and } m=m_k, \\
	0, & \quad  \text{otherwise}.
\end{aligned}
\right.
$$
Notice that this sequence of Fourier coefficients define an admissible distribution $f \in \mathbb{E}$, because if $\Delta(\xi)_m=0$, then $\xi \neq \xi_k$ and $\xi \neq \bar{\xi_k}$, for all $k\in \mathbb{N}$, by \eqref{counter}. Suppose that $Pu=f$ for some $u \in \mathcal{D}'(G)$. Then by \eqref{relation-sol2} we have
\begin{align}
\Delta(\xi_k)_{m_k} \widehat{u}(\xi_k)_{m_kn} &= (i\lambda_{m_k}(\xi_k)-\bar{q})\widehat{f}(\xi_k)_{m_kn}+p\overline{\widehat{{f}}(\bar{\xi_k})_{m_kn}}\\
 &= (i\lambda_{m_k}(\xi_k)-\bar{q})0+p\overline{	{\overline{p}}^{-1}}\\
 &=1,
\end{align}
for all $1 \leq n \leq d_{\xi_k}$. Hence,
$$
|\widehat{u}(\xi_k)_{m_kn}| = | \Delta(\xi_k)_{m_k}| ^{-1} \geq C^{-1} \jp{\xi_k}^k,
$$
which contradicts the fact that $u \in \mathcal{D}'(G)$. We have shown that there is no $u \in \mathcal{D}'(G)$ satisfying $Pu=f$, with $f\in \mathbb{E}$, therefore $P$ is not globally solvable, which concludes the proof.
\end{proof}
The next corollary is a direct consequence of Theorem \ref{gh-nec} and Theorem \ref{GS}.
\begin{corollary}
Assume that  $[\xi] \neq [\bar{\xi}],$
for all non-trivial $[\xi] \in \widehat{G}$. If $P$ is globally hypoelliptic  then $P$ is globally solvable.
\end{corollary}

Thus, we have a  class of examples of operators $P$ that are globally solvable by considering globally hypoelliptic operators defined on groups that satisfy the hypothesis about the  representations. Let us see now that the converse does not hold.
\begin{example}
We have seen in Example \ref{notgh} that the operator
$$
Pu(t,x):= \partial_tu(t,x) + \partial_0u(t,x) - \tfrac{\sqrt{5}}{2}iu(t,x) - \overline{u(t,x)}, \quad (t,x)\in \mathbb{T}^1\times {\SU},
$$
is not globally hypoelliptic. However, this operator is globally solvable by condition 3 of Theorem \ref{GS}. Indeed, if
$$
\left|(\tau+m)^2-\left(\frac{5}{4}-1\right)\right|=\left|(\tau+m)^2-\frac{1}{4}\right|\neq0,
$$
for some $\tau \in \mathbb{Z}$ and $m \in \frac{1}{2}\mathbb{N}_0$, then 
$$
\left|(\tau+m)^2-\frac{1}{4}\right| > \frac{1}{4},
$$
which implies condition 3 of Theorem \ref{GS}.
\end{example}
Finally, let us see an example of an operator that is not globally solvable.
\begin{example}
Let $G=\mathbb{T} \times \SU$ and $\alpha \in \mathbb{R}$ be an irrational Liouville number. Consider the operator
$$
Pu(t,x):= \partial_tu(t,x) + \alpha \partial_0u(t,x) - iu(t,x) - \overline{u(t,x)}, \quad (t,x)\in \mathbb{T}^1\times {\SU}.
$$
Notice that $P$ does not satisfy neither condition 1 nor 2 of Theorem \ref{GS}. Assume that $P$ satisfies condition 3, then in particular there exists $M>0$ such that
$$
|(\tau + \alpha \ell)^2| \geq \jp{(\tau,\ell)}^{-M},
$$
for $\tau \in \mathbb{Z}$ and $\ell \in \mathbb{N}_0$, $(\tau,\ell) \neq (0,0)$. However, this implies that $\alpha$ is not a Liouville number, so the operator $P$ is not globally solvable.
\end{example}

In \cite{BDM14} the authors proved that $P$ is globally hypoelliptic if and only if $P$ is solvable. However, since we are using a different notion of solvability, we do not have this equivalence here, as seen in the previous example.

By the expressions \eqref{solutiondeltaneq0} and \eqref{solutiondeltaeq0} that we have obtained for the solution of the equation $Pu=f$, we have the following corollary:
\begin{corollary}
Assume that  $[\xi] \neq [\bar{\xi}],$
for all non-trivial $[\xi] \in \widehat{G}$. If $P$ is globally solvable, then for every admissible smooth function $f\in \mathbb{E}$, there exists $u \in C^\infty(G)$ such that $Pu=f$.
\end{corollary}
\section{The operator $P$ on $\textrm{SU}(2)$}\label{operatoronsu2}
We have supposed on Proposition \ref{lemma-nec}, Theorem \ref{gh-nec} and Theorem \ref{GS} that every non-trivial irreducible unitary representation of the groups is not self-dual. This hypothesis came from the technique that we have used to prove these results, which was based on the torus case. In this section, we present these results for $\textrm{SU}(2)$, where we have 
$$
[\xi]=[\bar{\xi}],
$$
for all $[\xi] \in \widehat{\textrm{SU}(2)}$. 

Recall from Example \ref{ex-su2} that $\widehat{\textrm{SU}(2)}$ consists of equivalence classes $[\textsf{t}^\ell]$ of continuous irreducible unitary representations
$\textsf{t}^\ell:\textrm{SU}(2)\to \mathbb{C}^{(2\ell+1)\times (2\ell+1)}$, $\ell\in\frac12\mathbb{N}_{0}$ and we use the standard convention of enumerating the matrix elements
$\textsf{t}^\ell_{mn}$ of $\textsf{t}^\ell$ using indices $m,n$ ranging between
$-\ell$ to $\ell$ with step one, i.e. we have $-\ell\leq m,n\leq\ell$
with $\ell-m, \ell-n\in\mathbb{N}_{0}.$ For $\ell \in \tfrac{1}{2} \mathbb{N}_0$ we have
$$
\jp{\ell} := \jp{\textsf{t}^\ell} = \sqrt{1+\ell(\ell+1)}. 
$$

Consider the operator
\begin{equation}\label{operatorPsu2}
Pu:=\partial_0 u - qu-p\bar{u},
\end{equation}
where $\partial_0$ is the neutral operator, and $q,p \in \mathbb{C}$, with $p \neq 0$. As mentioned in Example \ref{ex-su2}, there is no loss of generality assuming that the vector field is $\partial_0$ because all vector fields define on $\textrm{SU}(2)$ can be conjugated to $\partial_0$. We have that
$$
\sigma_{\partial_0}(\ell)_{mn}=im\delta_{mn},
$$
for all $\ell\in \frac{1}{2} \mathbb{N}_0$. Hence, we obtain
\begin{align*}
\widehat{Pu}(\ell)_{mn} &= (im-q)\widehat{u}(\ell)_{mn}- p \widehat{\bar{u}}(\ell)_{mn}\\
&=(im-q)\widehat{u}(\ell)_{mn} - p \overline{\widehat{{u}}(\bar{\ell})_{mn}},
\end{align*}
where $\bar{\ell} := \overline{\textsf{t}^\ell}$. By the properties of representation on $\textrm{SU}(2)$, we have that
$$
\overline{\textsf{t}^\ell(x)}_{mn}=(-1)^{m-n}\textsf{t}^\ell(x)_{-m-n}.
$$
Hence,
\begin{align*}
\widehat{u}(\bar{\ell})_{mn} = \int_{\textrm{SU}(2)} u(x)\overline{\overline{\textsf{t}^\ell(x)_{nm}}} \textrm{d}x &= \int_{\textrm{SU}(2)} u(x) \overline{(-1)^{m-n}\textsf{t}^\ell(x)_{-m-n}} \textrm{d}x\\&= (-1)^{m-n}\widehat{u}(\ell)_{-m-n}
\end{align*}
Therefore,
\begin{equation}\label{exem-su2}
\widehat{Pu}(\ell)_{mn}=(im-q)\widehat{u}(\ell)_{mn}- p(-1)^{m-n}\overline{\widehat{{u}}({\ell})_{-m-n}}
\end{equation}
From \eqref{delta-def} and \eqref{relation-sol} we obtain
\begin{align}
\Delta(\ell)_m \widehat{u}(\ell)_{mn}&=(im-\bar{q})\widehat{Pu}(\ell)_{mn} +p\overline{\widehat{Pu}(\bar{\ell})_{mn}}\nonumber
\\
&=(im-\bar{q})\widehat{Pu}(\ell)_{mn} +p(-1)^{m-n}\overline{\widehat{Pu}(\ell)_{-m-n}},\label{delta-su2}
\end{align}
where 
$$
\Delta(\ell)_m = -m^2+|q|^2-|p|^2-i2m\textrm{Re}(q).
$$
\begin{proposition}
	Either $\Delta(\cdot)_\cdot$ vanishes for infinitely many representations or $\Delta(\cdot)_\cdot$ never vanishes.
\end{proposition}
\begin{proof} 
Let $\kappa \in \frac{1}{2}\mathbb{N}_0$ and $-\kappa \leq r \leq \kappa$ such that $\kappa-r \in \mathbb{Z}$. Hence, we obtain
$$
\Delta(\kappa)_r = -r^2+|q|^2-|p|^2-i2r\textrm{Re}(q)= \Delta(\kappa+j)_r,
$$
for all $j \in \mathbb{N}$, because $-(\kappa+j) \leq -\kappa \leq r \leq \kappa \leq \kappa+j$ and $(\kappa+j)-r = (\kappa-r)+j \in \mathbb{Z}$. Thus, if $\Delta(\kappa)_r=0$ then $\Delta(\kappa+j)_r=0$, for all $j\in \mathbb{N}$. This means that once $\Delta(\cdot)_\cdot$ vanishes, it will vanish for infinitely many other representations. 
\end{proof}

\begin{proposition}\label{delta-eq}
	We have $\Delta(\ell)_m \neq 0$, for all $\ell \in \frac{1}{2}\mathbb{N}_0$, $-\ell \leq m \leq \ell$, $\ell -m \in \mathbb{N}_0$ if and only if one of the following statements holds:
\begin{enumerate}
\item $|p|>|q|$;
\item $|p|<|q|$ and $\mathrm{Re}(q) \neq 0$;
\item $\exists M>0$ such that
$$
\jp{\ell} \geq M \implies 	\Big|m^2-(|q|^2-|p|^2)\Big| \geq \jp{\ell}^{-M},
$$
for all $-\ell\leq m \leq \ell$.
\end{enumerate}
\end{proposition}
\begin{proof} We have seen in the proof of Theorem \ref{thm-suf} that both conditions 1. and 2. imply that $\Delta(\ell)_m \neq 0$, for all $\ell \in \frac{1}{2}\mathbb{N}_0$. Notice that condition 3. says that $\mathrm{Re}(\Delta(\ell)_m) \neq 0$, for all $\jp{\ell} \geq M$. In particular we obtain by the last proposition that $\mathrm{Re}(\Delta(\ell)_m) \neq 0$, for all $\ell \in \frac{1}{2}\mathbb{N}_0$, which complete the proof of the sufficiency. 
	
Notice that condition 3 is equivalent to 
\begin{enumerate}
	\item[3'.] $\sqrt{|q|^2-|p|^2} \notin \frac{1}{2}\mathbb{N}_0$.
\end{enumerate} 

Indeed, if 3. holds then  $\mathrm{Re}(\Delta(\ell)_m) \neq 0$, for all $\ell \in \frac{1}{2}\mathbb{N}_0$, which implies 3'. On the other hand, if 3'. holds than there exists $C>0$ such that
$$
|m^2-(|q|^2-|p|^2)|\geq C,
$$
for all $m \in \frac{1}{2}\mathbb{Z}$, which implies 3.

Let us prove now the necessity. If none of these 3 conditions holds, then we have one of the following possibilities:
	\begin{enumerate}[(a)]
	\item $|p|<|q|, \mathrm{Re}(q) =0$, and 3'. does not hold;
	\item $|p|=|q|$.
\end{enumerate}
Notice that the case $|p|=|q|$ implies that 3'. does not hold. If we are in the case (a), then we have that $s=\sqrt{|q|^2-|p|^2} \in \frac{1}{2}\mathbb{N}_0$. Hence, 
$$
\Delta(s)_s = -s^2+|q|^2-|p|^2 - i2s\textrm{Re}(q) = 0,
$$
which contradicts the hypothesis. 
For the case (b), just notice that $\Delta(\ell)_0=0$, for all $\ell \in \mathbb{N}$. 
\end{proof}

Now we can give the sufficient and necessary condition for the global hypoellipticity of the operator $P$ defined on $\SU$.
\begin{theorem}
The operator $P$ is  globally hypoelliptic if and only if $\Delta(\ell)_m \neq 0$, for all $\ell \in \frac{1}{2}\mathbb{N}_0$, $-\ell \leq m \leq \ell$, $\ell -m \in \mathbb{N}_0$.
\end{theorem}

\begin{proof}
	The sufficiency follows from Theorem \ref{thm-suf} and Proposition \ref{delta-eq}. 
	
	Assume now that $\Delta(\ell)_m=0$ for some $\ell \in \frac{1}{2}\mathbb{N}_0$, $- \ell \leq m \leq \ell, \ell-m\in \mathbb{N}_0$. Let us proceed by contradiction to prove the necessity. Following the proof of Proposition \ref{delta-eq}, we have one of the following possibilities:
	\begin{enumerate}[(a)]
		\item $|p|<|q|, \mathrm{Re}(q) =0$, and $\sqrt{|q|^2-|p|^2} \in \frac{1}{2}\mathbb{N}_0$;
		\item $|p|=|q|$.
	\end{enumerate}
Let us construct a singular solution for $Pu=0$ for both cases. For the case (a), let $s:= \sqrt{|q|^2-|p|^2}$ and consider
$$
\widehat{u}(\ell)_{mn} =
\left\{
\begin{array}{ll}
0,& \text{if } m^2 \neq s^2 \\
is-\bar{q}, & \text{if } m=s\\
p(-1)^{-s-n}, & \text{if } m=-s.
\end{array}
\right. 
$$
Clearly we have $u \in \mathcal{D}'(\SU)$. Moreover, this distribution does not come from a smooth function because  $|\widehat{u}(s+j)_{-sn}| = |p| \neq 0$, for all $j \in \mathbb{N}_0$. By \eqref{exem-su2} we have
$$
\widehat{Pu}(\ell)_{mn}=(im-q)\widehat{u}(\ell)_{mn}- p(-1)^{m-n}\overline{\widehat{{u}}({\ell})_{-m-n}}
$$

Hence, if $m^2 \neq s^2$ we obtain $\widehat{Pu}(\ell)_{mn}=0$. For $m=s$ we have 
\begin{align*}
\widehat{Pu}(\ell)_{sn}&=(is-q)\widehat{u}(\ell)_{sn}- p(-1)^{s-n}\overline{\widehat{{u}}({\ell})_{-s-n}}\\
&=(is-q)(is-\bar{q}) -p(-1)^{s-n}\bar{p} (-1)^{-s+n} \\
&=-s^2+|q|^2-2is\text{Re}(q)-|p|^2\\
&=0.
\end{align*}

For $m=-s$ we have
\begin{align*}
	\widehat{Pu}(\ell)_{-sn}&=(-is-q)\widehat{u}(\ell)_{-sn}- p(-1)^{-s-n}\overline{\widehat{{u}}({\ell})_{s-n}}\\
	&=(-is-q)p(-1)^{-s-n}-p(-1)^{-s-n}(-is-q)\\
	&=0.
\end{align*}

Therefore $Pu=0$, which is a contradiction by the hypothesis about the global hypoellipticity of $P$. 

For the case (b), consider the following sequence of Fourier coefficients:
$$
\widehat{u}(\ell)_{mn} =
\left\{
\begin{array}{ll}
	0,& \text{if } m \neq 0\\
	0,& \text{if } m=0 \text{ and } n=0\\
	-\bar{q}, & \text{if }  m=0 \text{ and }n>0\\
	p(-1)^{-n} & \text{if } m=0 \text{ and } n<0.
\end{array}
\right. 
$$

Similarly to the previous case we have $u \in \mathcal{D}'(\SU) \setminus C^\infty(\SU)$. Again by \eqref{exem-su2}, when $m \neq 0$ we have $\widehat{Pu}(\ell)_{mn}=0$. For $m=0$ and $n>0$ we obtain
\begin{align*}
	\widehat{Pu}(\ell)_{0n}&=(-q)\widehat{u}(\ell)_{0n}- p(-1)^{-n}\overline{\widehat{{u}}({\ell})_{0-n}}\\
	&=(-q)(-\bar{q})-p(-1)^{-n}\bar{p}(-1)^{n}\\
	&=|q|^2-|p|^2 \\
	&=0.
\end{align*}
For $m=0$ and $n<0$, we have
\begin{align*}
	\widehat{Pu}(\ell)_{0n}&=(-q)\widehat{u}(\ell)_{0n}- p(-1)^{-n}\overline{\widehat{{u}}({\ell})_{0-n}}\\
	&=(-q)p(-1)^{-n}-{p}(-1)^{-n}(-q)\\
	&=0.
\end{align*}
Finally, for $m=n=0$ we have
$$
	\widehat{Pu}(\ell)_{00}=(-q)\widehat{u}(\ell)_{00}- p\overline{\widehat{{u}}({\ell})_{00}}=0.
$$
We conclude that $Pu=0$, which contradicts the global hypoellipticity of $P$. 
\end{proof}
\begin{example}
The following operators on ${\SU}$ are globally hypoelliptic:
\begin{enumerate}
	\item $P_1u:= \partial_0u - u - 2\overline{u}, \quad (|p|>|q|)$;
		\item $P_2u:= \partial_0u -2u - i\overline{u}, \quad (|p|<|q| \text{ and } {\textrm{Re}}(q) \neq 0) $;
			\item $P_3u:= \partial_0u - 3i - \overline{u}, \quad (\sqrt{|q|^2-|p|^2} \notin \frac{1}{2}\mathbb{N}_0)$.
\end{enumerate}
The following operators on ${\SU}$  are not globally hypoelliptic:
\begin{enumerate}
	\item $P_au:= \partial_0u -5iu-4\overline{u}, \quad (|p|<|q|, {\textrm{Re}}(q)=0, \text{ and } \sqrt{|q|^2-|p|^2} \in \frac{1}{2}\mathbb{N}_0)$;
	\item $P_bu:= \partial_0u - (1+i)u - \sqrt{2} i\overline{u}, \quad (|p|=|q|)$.
\end{enumerate}
\end{example}

Let us investigate the global solvability of the operator $P$ on $\SU$. By \eqref{delta-su2} we have
$$
\Delta(\ell)_m \widehat{u}(\ell)_{mn}=(im-\bar{q})\widehat{Pu}(\ell)_{mn} +p(-1)^{m-n}\overline{\widehat{Pu}(\ell)_{-m-n}},
$$
so the space of the admissible distributions for the operator $P$ on $\SU$ is given by
$$
\mathbb{E}:= \{ f \in \mathcal{D}'(G) ;  \ \Delta(\ell)_m=0 \implies (im-\bar{q})\widehat{f}(\ell)_{mn} +p(-1)^{m-n}\overline{\widehat{f}(\ell)_{-m-n}}=0, \ -\ell \leq n \leq \ell \}.
$$
Hence, we say that the operator $P$ is globally solvable if for every $f \in \mathbb{E}$ there exists $u \in \mathcal{D}'(\SU)$ such that $Pu=f$.

Given $f \in \mathbb{E}$, notice that when $\Delta(\ell)_m \neq 0$ we can define
\begin{equation}\label{solution-su2}
\widehat{u}(\ell)_{mn} = \Delta(\ell)_m^{-1}\left((im-\bar{q})\widehat{f}(\ell)_{mn} +p(-1)^{m-n}\overline{\widehat{f}(\ell)_{-m-n}}\right)
\end{equation}
and we get $\widehat{Pu}(\ell)_{mn}  = \widehat{f}(\ell)_{mn}$. When $\Delta(\ell)_m \neq 0$ we have at least one of the following:
\begin{itemize}
	\item $\textrm{Im}(\Delta(\ell)_m) \neq 0$.
\end{itemize}

In this case, $2m\textrm{Re}(q) \neq 0$, so $\text{Re}(q) \neq 0$ and 
$$
|\Delta(\ell)_m| \geq | \textrm{Im}(\Delta(\ell)_m)| = |2m\textrm{Re}(q) | \geq |\text{Re}(q)| = C_1 >0 
$$
because $m \in \frac{1}{2}\mathbb{N}_0$ and $m \neq 0$. 
\begin{itemize}
	\item $\textrm{Re}(\Delta(\ell)_m) \neq 0$.
\end{itemize}
Since $m \in \frac{1}{2}\mathbb{N}_0$, there exists $C_2>0$ such that $|-m^2+|q|^2-|p|^2| >C_2$, whenever $m \in \frac{1}{2}\mathbb{N}_0$ satisfying $-m^2+|q|^2-|p|^2 \neq 0$. Hence,
$$
|\Delta(\ell)_m| \geq | \textrm{Re}(\Delta(\ell)_m)| = |-m^2+|q|^2-|p|^2 | \geq  C_2 >0.
$$
Thus, if $\Delta(\ell)_m \neq 0$, there exists $C>0$, which does not depend neither on $\ell$ nor on $m$, such that 
$$
|\Delta(\ell)_m| >C.
$$
Therefore the estimates for $\widehat{u}(\ell)$ defined on \eqref{solution-su2} are similar to the estimates for $\widehat{f}(\ell)$.

We still need to analyze the case where $\Delta(\ell)_m=0$. First, assume that $\Delta(\ell)_s=0$ with $s \neq 0$. Notice that we also have $\Delta(\ell)_{-s}=0$, so there is no loss of generality to suppose that $s>0$. Define
\begin{equation}\label{solutionm0}
	\left\{ 
	\begin{aligned}
		\widehat{u}(\ell)_{sn} &=0; \\
		\widehat{u}(\ell)_{-sn} & = -\bar{p}^{-1}(-1)^{s+n}\overline{\widehat{f}(\ell)_{s -n}},
	\end{aligned}
	\right.
\end{equation}
Let us verify that $\widehat{Pu}(\ell)_{sn}= \widehat{f}(\ell)_{sn}$ and $\widehat{Pu}(\ell)_{-sn}= \widehat{f}(\ell)_{-sn}$, for all $-\ell \leq n \leq \ell.$ We have
\begin{align*}
	\widehat{Pu}(\ell)_{sn} &= (is-q)\widehat{u}(\ell)_{sn} -p(-1)^{s-n} \overline{\widehat{u}(\ell)_{-s-n}}\\
	 &= (is-q)0 +p(-1)^{s-n} \overline{\bar{p}^{-1}(-1)^{s-n}\overline{\widehat{f}(\ell)_{sn}}}\\
	 &=\widehat{f}(\ell)_{s n}
\end{align*}
and
\begin{align*}
	\widehat{Pu}(\ell)_{-sn} &= (-is-q)\widehat{u}(\ell)_{-sn} -p(-1)^{-s-n} \overline{\widehat{u}(\ell)_{s-n}}\\
&= (-is-q)(-\bar{p}^{-1}(-1)^{s+n}\overline{\widehat{f}(\ell)_{s -n}}) +p(-1)^{s-n}0 \\
&=\overline{(is-\bar{q})(-{p}^{-1}(-1)^{s+n}{\widehat{f}(\ell)_{s -n}})} \\
&= \widehat{f}(\ell)_{-sn},
\end{align*}
where the last step follows from the fact that $f \in \mathbb{E}$. Finally, let us consider the case where $\Delta(\ell)_0=0$. Here, we have $|q|^2=|p|^2$, which implies that $q \neq 0$. Define
\begin{equation}\label{solution-su2 2}
\widehat{u}(\ell)_{0n} = 
\left\{
\begin{array}{ll}
	0, & \text{if } n>0;\\
	-(2q)^{-1} \widehat{f}(\ell)_{00},& \text{if } n=0;\\
	-\bar{p}^{-1}(-1)^n\overline{\widehat{f}(\ell)_{0-n}}, & \text{if } n<0.
\end{array}
\right.
\end{equation}
When $n>0$, we have
\begin{align*}
	\widehat{Pu}(\ell)_{0n} &= -q\widehat{u}(\ell)_{0n} -p(-1)^{-n} \overline{\widehat{u}(\ell)_{0-n}}\\
	&= (-q)0 +p(-1)^{-n} \overline{\bar{p}^{-1}(-1)^{-n}\overline{\widehat{f}(\ell)_{0n}}}\\
	&=\widehat{f}(\ell)_{0n}
\end{align*}
On the other hand, when $n<0$, we obtain
\begin{align*}
	\widehat{Pu}(\ell)_{0n} &= -q\widehat{u}(\ell)_{0n} -p(-1)^{-n} \overline{\widehat{u}(\ell)_{0-n}}\\
&= -q(	-\bar{p}^{-1}(-1)^n\overline{\widehat{f}(\ell)_{0-n}})+p(-1)^{-n}0\\
&=\overline{\bar{q}({p}^{-1}(-1)^n\overline{\widehat{f}(\ell)_{0-n}})}\\
&=\widehat{f}(\ell)_{0n},
\end{align*}
because $f \in \mathbb{E}$. Similarly, for $n=0$ we have
\begin{align*}
	\widehat{Pu}(\ell)_{00} &= -q\widehat{u}(\ell)_{00} -p\overline{\widehat{u}(\ell)_{00}}\\
	&= (-q)(-(2q)^{-1} \widehat{f}(\ell)_{00}) +p\overline{(-(2q)^{-1} \widehat{f}(\ell)_{00})}\\
	&=\widehat{f}(\ell)_{00}.
\end{align*}
Therefore given $f \in \mathbb{E}$, we were able to construct $u$ such that $\widehat{Pu}(\ell)=\widehat{f}(\ell)$, for all $\ell \in \frac{1}{2}\mathbb{N}_0$. Moreover, by the expressions \eqref{solution-su2}, \eqref{solutionm0} and \eqref{solution-su2 2}, we conclude that $u$ has  decay (and growth) similar to $f$. Hence, we have proved the following theorem:

\begin{theorem}
The operator $P: \mathcal{D}'(\emph{\SU}) \to \mathcal{D}'(\emph{\SU})$ given by
$$
Pu:=\partial_0 u - qu-p\bar{u},
$$
with $q, p \in \mathbb{C}$, $p \neq 0$, is globally solvable. Moreover, if $f$ is an admissible smooth function, then there exists $u \in C^\infty(\emph{\SU})$ such that $Pu=f$.
\end{theorem} 
\section*{Acknowledgements}
The author would like to thank Alexandre Kirilov and Ricardo Paleari da Silva for comments and suggestions.
 \section*{Conflict of interest}

 The authors declare that they have no conflict of interest.
\bibliographystyle{abbrv}
\bibliography{biblio}
\end{document}